\documentclass[12pt,a4paper]{article}
\usepackage[utf8]{inputenc}
\usepackage{tikz}
\usetikzlibrary{arrows,automata}
\tikzset{every state/.style={minimum size=0pt}}
\usepackage{smartdiagram}  
\usepackage{algorithm}
\usepackage{algcompatible}
\usepackage{graphicx}
\usepackage{xcolor}
\usepackage[algo2e,linesnumbered,ruled,vlined]{algorithm2e}
\usepackage[english]{babel}
\usepackage{subfig}
\usepackage{animate}
\usepackage{authblk}
\usepackage[margin=1in]{geometry}
\usepackage[T1]{fontenc}
\usepackage{amssymb, amsthm}
\usepackage{amsmath}
\usepackage{mathtools}
\usepackage{relsize}
\usepackage{todonotes}
\usepackage{hyperref}

\newtheorem{theorem}{Theorem}
\newtheorem{lemma}[theorem]{Lemma}
\newtheorem{corollary}[theorem]{Corollary}

\newtheorem{observation}[theorem]{Observation}
\newtheorem{proposition}[theorem]{Proposition}
\newtheorem{example}[theorem]{Example}
\newtheorem{definition}[theorem]{Definition}
\newtheorem{remark}[theorem]{Remark}

\newcommand{\defproblem}[3]{
  \vspace{3mm}
\noindent\fbox{
  \begin{minipage}{.97\textwidth}
  \begin{tabular*}{\textwidth}{@{\extracolsep{\fill}}lr} \textsc{#1} \\ \end{tabular*}
  {\bf{Input:}} #2  \\
  {\bf{Question:}} #3
  \end{minipage}
  }
  \vspace{2mm}
}

\begin{document}
\title{Bounds and Hardness Results for Conflict-free Choosability}

\author{ Shiwali Gupta, Rogers Mathew}
\affil {Department of Computer Science and Engineering, Indian Institute of Technology Hyderabad. \{cs21resch11002, rogers\}@iith.ac.in}

\date{}
\maketitle
\textbf{keywords:}{Conflict-free coloring \and list conflict-free coloring, conflict-free choosability, choice number, maximum degree, minimum degree, hardness. } 

\begin{abstract}
A \emph{(partial) conflict-free coloring} of a hypergraph $\mathcal{H}$ is an assignment of colors to (a subset of) the vertex set of $\mathcal{H}$ such that every hyperedge in $\mathcal{H}$ has a vertex whose color is distinct from every other vertex in that hyperedge. The minimum number of colors required for such a coloring is known as the \emph{(partial) conflict-free chromatic number} of $\mathcal{H}$. It is easy to see that the conflict-free chromatic number of a hypergraph is at most its partial conflict-free chromatic number plus one. Conflict-free coloring has also been studied on the open/closed neighborhood hypergraphs of a given graph under the name open/closed neighborhood conflict-free coloring. In this paper, we study partial and full list variants of conflict-free coloring where, for every vertex $v$, we are given a list of admissible colors $L_v$ such that $v$ is allowed to be colored only from $L_v$.
\\
 
It was shown by Pach and Tardos [Combinatorics, Probability and Computing, 2009] that for any constant $\epsilon > 0$, the closed-neighborhood conflict-free chromatic number of a graph $G$ is at most $O(\ln^{2 + \epsilon}\Delta)$, where $\Delta$ represents the maximum degree of $G$.
Later, Glebov, Szabó, and Tardos [Combinatorics, Probability and Computing, 2014] showed that there exist graphs $G$ that require $\Omega(\ln^2\Delta)$ colors for a closed neighborhood conflict-free coloring. Bhyravarapu, Kalyanasundaram, and Mathew [Journal of Graph Theory, 2021] bridged the gap between the upper and the lower bound. They showed that the closed-neighborhood conflict-free chromatic number of any graph $G$ is at most $O(\ln^2 \Delta)$. 
\\

In this paper, we extend the $O(\ln^2 \Delta)$ upper bound to the partial list variant of the closed-neighborhood conflict-free chromatic number.  Further, we establish computational complexity results concerning the list open/closed-neighborhood conflict-free chromatic numbers.

\end{abstract}

\section{Introduction}
\subsection{Conflict-free coloring}
\label{subsec_CF_coloring}

A \emph{partial conflict-free coloring} (or \emph{CF}$^*$ \emph{coloring}) of a hypergraph $\mathcal{H}=(V,\mathcal{E})$ using $k$ colors is an assignment $f: V'\rightarrow \{1, \ldots, k \}$, where $V'\subseteq  V$, such that every $E \in \mathcal{E}$ contains a point whose color is distinct from that of every other point in $E$. If $V' = V$, then we call $f$ a \emph{conflict-free coloring} (or \emph{CF coloring}) of $\mathcal{H}$. The minimum $k$ for which there is a (partial) CF coloring of $\mathcal{H}$ using $k$ colors is called the (resp., \emph{partial}) \emph{conflict-free chromatic number} of $\mathcal{H}$. We shall use $\chi_{CF}(\mathcal{H})$ (resp., $\chi_{CF}^*(\mathcal{H})$) to denote (resp., partial) conflict-free chromatic number of $\mathcal{H}$. Giving a new unused color to all the uncolored vertices in a partial CF coloring of $\mathcal{H}$ gives a CF coloring of $\mathcal{H}$. The observation below captures this.

\begin{observation}
 \label{obv_partial_vs_full_CF}   
 $\chi_{CF}(\mathcal{H}) \leq \chi^*_{CF}(\mathcal{H}) + 1$. 
\end{observation}

The notion of conflict-free coloring is well studied with respect to \emph{open/closed neighborhood hypergraphs of graphs} (the edge set of such a hypergraph is the set of open/closed neighborhoods of each vertex of the graph under consideration).
Let $V(G)$ and $E(G)$ denote the vertex set and edge set of a graph $G$, respectively. For a vertex $v \in V(G)$, the set of neighbors of $v$ in $G$ is called the \emph{open neighborhood} of $v$. It is denoted by $N_G(v)$. The \emph{closed neighborhood} of $v$ is defined as $\{v\} \cup N_G(v)$. It is denoted by $N_G[v]$. 

A \emph{partial conflict-free open neighborhood coloring} (or \emph{CFON$^*$ coloring}) of $G$  is an assignment of colors to $V' \subseteq V(G)$ such that every vertex in $G$ sees a uniquely colored vertex  in its open neighborhood. If $V' = V(G)$, then we call it a \emph{CFON coloring}. The minimum number of colors required for a CFON$^*$ coloring (resp., CFON coloring) of $G$ is called the \emph{CFON$^*$ chromatic number} (resp., \emph{CFON chromatic number}) of $G$, denoted by $\chi^*_{ON}(G)$ (resp., $\chi_{ON}(G)$). 
Analogously, by replacing `open neighborhood' with `closed neighborhood' in the above definitions, we define CFCN$^*$ coloring, CFCN coloring, CFCN$^*$ chromatic number (denoted $\chi^*_{CN}(G)$), and  CFCN chromatic number (denoted $\chi_{CN}(G)$). 

Observation \ref{obv_partial_vs_full_CF} implies that for any graph $G$, $\chi_{ON}(G) \leq \chi^*_{ON}(G) + 1$ and $\chi_{CN}(G) \leq \chi^*_{CN}(G) + 1$.
The following inequality connects the CF open and closed chromatic numbers of a graph with each other.

\begin{proposition}[Inequality 1.3 in \cite{pach2009conflict}]
\label{prop:CFCNON}
(i) $\chi_{CN}(G) \leq 2 \chi_{ON}(G)$, and
(ii) $\chi^*_{CN}(G) \leq 2 \chi^*_{ON}(G)$. 
\end{proposition}

The example given below shows that there are graphs whose CFON chromatic number is arbitrarily large compared to its CFCN chromatic number. 

\begin{example}[\cite{pach2009conflict}]
\label{ex:K_n subdivided}
Let $K_n$ denote the complete graph on $n$ vertices. Let $K_n^{1/2}$ denote the graph obtained by subdividing every edge of $K_n$ exactly once. Then, it is known that (i) $\chi_{ON}(K_n^{1/2}) = n$, and (ii) $\chi_{CN}(K_n^{1/2}) = 2$. 
\end{example}

Partial conflict-free coloring on open/closed neighborhood is NP-complete for planar \cite{abel2018conflict} and chordal graphs \cite{bhyravarapu2021conflictinter, bhyravarapu2025conflict}, efficiently solvable for outerplanar \cite{abel2018conflict}, and interval graphs \cite{bhyravarapu2021conflictinter, bhyravarapu2022conflict}. 
It was shown in \cite{gargano2015complexity} that deciding whether a graph has a conflict-free coloring with two colors is NP-complete for both closed and open neighborhoods. 
Moreover, CFCN variant is hard to approximate within a factor less than $3/2$, unless $\mathsf{P} = \mathsf{NP}$. The CFON variant cannot be approximated within a factor $n^{1/2 - \varepsilon}$ for any $\varepsilon > 0$, unless $\mathsf{P} = \mathsf{NP}$.

Conflict-free coloring was first studied by Even et al. \cite{Even2002} in 2004 and was further studied in many contexts (see \cite{cheilaris2009conflict, glebov2014conflict, pach2009conflict, bhyravarapu2021short, dkebski2022conflict, bhyravarapu2022conflict, LEVTOV20091521, ElbassioniM06,pach2003conflict, har2003conflict, alon2006conflict}).  See Smorodinsky \cite{smorosurvey} for a survey on conflict-free coloring and its applications. Conflict-free coloring has found applications in frequency assignment problems for cellular networks, in battery consumption-related optimization problems in sensor networks, in vertex ranking problems that find applications in VLSI design and operations research, in pliable index coding problem in coding theory \cite{krishnan2022pliable} etc. 

The study on conflict-free coloring was originally motivated by its application in the frequency assignment problem for cellular networks.
There are two types of nodes in a cellular network: base stations and mobile clients. Base stations serve as the network's backbone and have a fixed position. Mobile clients are connected to the base station. The client and base station are connected by a radio link. Each base station has a fixed frequency to transmit or to enable a radio link for the clients. Clients are continuously scanning the frequencies in search of base stations with a good range. 
When a client is in the range of more than one base station, then mutual interference occurs if the same frequency is assigned to these base stations which will make the links noisy. The primary challenge in the frequency assignment problem in cellular networks is allocating frequencies to base stations in a way that ensures each client is served by some base station whose frequency is distinct from that of the other base stations in its range. The goal is to minimize the number of assigned frequencies because frequencies are expensive and limited; it is preferable to have a scheme that reuses frequencies wherever possible.

We can formulate the frequency assignment problem as a hypergraph coloring problem. Let $\mathcal{H} = (V, \mathcal{E})$ be a hypergraph, where vertices are the given set of base stations and each hyperedge denotes the set of base stations in the range of each client. Thus, we try to find the minimum number of colors required to color the vertices such that every hyperedge has a vertex with a color distinct from the color of all the other vertices in it.  Research on conflict-free coloring was performed under the assumption that any color from a global range of colors can be used and the goal was to minimize the number of colors used. Assume that the wireless network's base stations are further limited to using a subset of the available frequencies. As a result, different base stations may have access to different subsets of frequencies. Studying the list variant of conflict-free coloring is relevant in this context. In this variant, each vertex has a list of colors attached to it. Each vertex receives a color from its list. In the next section, we formally define this notion.

\subsection{List conflict-free coloring}
\label{sec list Cf coloring}
Let $\mathcal{H} = (V, \mathcal{E})$ be a hypergraph, and let $k$ be a positive integer. Let $\mathcal{L} = \{ L_v~:~v \in V\}$, where each $L_v$ is a list of colors. 

\begin{definition}[$k$-assignment]
Let $\mathcal{L} = \{ L_v~:~v \in V\}$ denote an assignment of a list of admissible colors to each vertex of $\mathcal{H}$. We say $\mathcal{L}$ is a \emph{$k$-assignment for $\mathcal{H}$} if $|L_v| = k$, for every $v \in V$. 
\end{definition}

\begin{definition}[$k$-CF$^*$-choosable, $k$-CF-choosable]
Given a list assignment $\mathcal{L} = \{ L_v~:~v \in V\}$, we say that $\mathcal{H}$ admits an \emph{$\mathcal{L}$-CF$^*$-coloring} if there exists a coloring $f: V' \rightarrow \bigcup_{v \in V'}L_v$, where $V' \subseteq V$, such that $f(v) \in L_v$, $\forall v \in V'$, and every hyperedge $E$ in $\mathcal{H}$ contains a point whose color is distinct from that of every other point in $E$. When $V' = V$, we say $\mathcal{H}$ admits an \emph{$\mathcal{L}$-CF-coloring}. We say that $\mathcal{H}$ is \emph{$k$-CF$^*$-choosable} (resp., \emph{$k$-CF-choosable}) if for every $k$-assignment $\mathcal{L}$, $\mathcal{H}$ admits an $\mathcal{L}$-CF\emph{$^*$}-coloring (resp., $\mathcal{L}$-CF-coloring). 
\end{definition}

\begin{definition}[CF$^*$ choice number, CF choice number]
\label{def:CF_list}  
The minimum $k$ for which $\mathcal{H}$ is $k$-CF\emph{$^*$}-choosable (resp., $k$-CF-choosable) is called the \emph{CF$^*$ choice number} (resp., \emph{CF choice number}) of $\mathcal{H}$. This is denoted by $ch^*_{CF}(\mathcal{H})$ (resp., $ch_{CF}(\mathcal{H})$).
\end{definition}

The following theorem due to Cheilaris, Smorodinsky, and  Sulovsk\'{y} \cite{cheilaris2011potential} gives us a relation between the CF chromatic number and CF choice number of a hypergraph $\mathcal{H}$.

\begin{theorem}
\cite{cheilaris2011potential}
\label{cf and l-cf relation}
For any hypergraph $\mathcal{H}$ with $n$ vertices,\\
 (i) $\chi_{CF}(\mathcal{H}) \le ch_{CF}(\mathcal{H}) \le  \chi_{CF}(\mathcal{H}) \cdot \ln n + 1$, and\\ (ii) $\chi^*_{CF}(\mathcal{H}) \le ch^*_{CF}(\mathcal{H}) \le  \chi^*_{CF}(\mathcal{H}) \cdot \ln n_1 + 1$, where $n_1 = \min \{ \mbox{no. of colored } \\ \mbox{vertices  under }  f~:~f \mbox{ is a CF$^*$ coloring of }  \mathcal{H} \mbox{ using only }  \chi^*_{CF}(\mathcal{H}) \mbox{ no. of colors} \}.$   
\end{theorem}

Similar to CFON and CFCN coloring, we define their list analogues below. Before we go into their definitions, let us look at the definition of the choice number of a graph $G$, denoted by $ch(G)$. Given $\mathcal{L} = \{ L_v~:~v \in V(G)\}$, we say $G$ admits \emph{$\mathcal{L}$-coloring}, if there exists a proper coloring $f: V(G) \rightarrow \bigcup_{v \in V(G)} L_v$ such that $f(v) \in L_v$, $\forall v \in V(G)$. We say $G$ is \emph{$k$-choosable} if $G$ admits an $\mathcal{L}$-coloring, for all $k$-assignments $\mathcal{L}$. The minimum $k$ for which $G$ is $k$-choosable is called the \emph{choice number} of $G$.

\begin{definition}[CFON$^*$ choice number, CFON choice number]
\label{defn_list_CFON}
Given a list assignment $\mathcal{L} = \{ L_v~:~v \in V(G)\}$ for a graph $G$, we say $G$ is \emph{$\mathcal{L}$-CFON$^*$-colorable}  if there exists a CFON\emph{$^*$} coloring $f: V' \rightarrow \bigcup_{v \in V'}L_v$, where $V' \subseteq V(G)$, such that $f(v) \in L_v$, $\forall v \in V'$. When $V' = V$, we say $G$ is \emph{$\mathcal{L}$-CFON-colorable}. We say that $G$ is $k$-CFON$^*$ (resp., $k$-CFON)-choosable if for every $k$-assignment $\mathcal{L}$, $G$ is $\mathcal{L}$-CFON$^*$ (resp., $\mathcal{L}$-CFON)-colorable. The minimum $k$ for which $G$ is $k$-CFON$^*$ (resp., $k$-CFON)-choosable is called the \emph{CFON$^*$}(resp., \emph{CFON})\emph{ choice number} of $G$. This is denoted by $ch^*_{ON}(G)$ (resp., $ch_{ON}(G)$). 
\end{definition}

Analogously, one can define the CFCN$^*$ choice number and CFCN choice number of a graph $G$ which will be denoted by $ch^*_{CN}(G)$ and  $ch_{CN}(G)$, respectively. 

To the best of our knowledge, the only paper that studies conflict-free choosability is due to Cheilaris, Smorodinsky, and  Sulovsk\'{y} in \cite{cheilaris2011potential}.
Using the potential method, they show that (see Theorem 2.5 \cite{cheilaris2011potential}) the CF choice number of a hereditarily $k$-colorable hypergraph on $n$ vertices is $O(k \ln n)$. Further, they show that the geometric hypergraphs (i) induced by a set of $n$ planar discs, and (ii) induced by $n$ points with respect to planar regions like discs, halfplanes, or intervals can be hereditarily colored using a constant number of colors. Thus, all these geometric hypergraphs have their CF choice number equal to $O(\ln n)$. As for a general hypergraph $\mathcal{H}$ with $m$ edges having a maximum degree of $\Delta$, they show:
\begin{eqnarray}
ch_{CF}(\mathcal{H}) \leq 1/2 + \sqrt{2m + 1/4} \label{Ineq_hypergraph_no_of_edges} \\
ch_{CF}(\mathcal{H}) \leq \Delta + 1 \label{Ineq_hypergraph_max_degree}
\end{eqnarray}  
Further in \cite{cheilaris2011potential}, the authors show that $ch_{CF}(\mathcal{H}_n) = \lfloor \ln_2 n \rfloor + 1$, where $\mathcal{H}_n$ is a `discrete interval hypergraph' on $n$ vertices.  For a graph $G$ on $n$ vertices, let $H_{G}^{path} := (V(G), \{S ~| ~S \text{ is the vertex set of a path in } G\})$. Theorem 4.1 in \cite{cheilaris2011potential} says that $ch_{CF} (H_{G}^{path}) \\= O(\sqrt{n})$.


\subsection{Some quick observations}
\label{sec quick obs}

Since in a classical list coloring of a graph $G$, each vertex sees its own color exactly once in its closed neighborhood, we have the following.

\begin{observation}
\label{obv_list_CFCN_vs_classical_list}
For any graph $G$,
$ch_{CN}(G) \le ch(G)$.
\end{observation}

Since CFON (or CFCN) coloring of a graph $G$ is a CF coloring of the open (or closed) neighborhood hypergraphs of $G$, we can extend Theorem \ref{cf and l-cf relation} to obtain the following proposition.

\begin{proposition}[\cite{cheilaris2011potential}]
\label{ON and l-ON relation}
    Let $G$ be a graph on $n$ vertices. Then, \\
    (i) $\chi_{ON}(G) \le ch_{ON}(G) \le  \chi_{ON}(G) \cdot \ln n + 1$, \\
    (ii) $\chi_{CN}(G) \le ch_{CN}(G) \le  \chi_{CN}(G) \cdot \ln n + 1$, \\
    (iii) $\chi_{ON}^*(G) \le ch_{ON}^*(G) \le  \chi_{ON}^*(G) \cdot \ln n_1 + 1$, where $n_1$ is the minimum number of colored vertices over all possible CFON$^*$ colorings of $G$ that use only $\chi_{ON}^*(G)$ colors. \\
    (iv) $\chi_{CN}^*(G) \le ch_{CN}^*(G) \le  \chi_{CN}^*(G) \cdot \ln n_1 + 1$, where $n_1$ is the minimum number of colored vertices over all possible CFCN$^*$ colorings of $G$ that use only $\chi_{CN}^*(G)$ colors. 
\end{proposition}

\begin{example}
\label{ex:star_graph}    
Let $k \geq 2$ be an  integer. Then, the complete bipartite graph $K_{1,k}$ satisfies (i) $\chi_{ON}(K_{1,k}) = ch_{ON}(K_{1,k}) = 2$, (ii) $\chi_{CN}(K_{1,k}) = ch_{CN}(K_{1,k}) = 2$, (iii) $\chi_{ON}^*(K_{1,k}) = ch_{ON}^*(K_{1,k}) = 1$, and (iv) $\chi_{CN}^*(K_{1,k}) = ch_{CN}^*(K_{1,k}) = 1$. 
\end{example}

Examples \ref{ex:partial_vs_full} and \ref{ex:partial_list_vs_partial} below imply that there are graphs $G$ on $n$ vertices for which $ch_{ON}(G) = \Omega(\chi_{ON}(G) \cdot \frac{\ln n}{\ln \ln n})$ and $ch^*_{ON}(G) = \Omega(\chi^*_{ON}(G) \cdot \frac{\ln n}{\ln \ln n})$. This makes the inequalities in statements (i) and (iii) of Proposition \ref{ON and l-ON relation} almost tight. The tightness of the other bounds mentioned in the above proposition is not known.  

Unlike in classical CF coloring (where $\chi_{CF}(\mathcal{H}) \leq \chi_{CF}^*(\mathcal{H}) + 1$), there are hypergraphs $\mathcal{H}$ for which $ch_{CF}(\mathcal{H})$ is arbitrarily larger than $ch^*_{CF}(\mathcal{H})$. See Example \ref{ex:partial_vs_full}. Example \ref{ex:partial_list_vs_partial} gives graphs whose CFON$^*$ chromatic number and CFON$^*$ choice number are far apart. 

\begin{example} 
\label{ex:partial_vs_full} 
Let $d>2$. Let $K_{d, d^d}$ denote the complete bipartite graph with bipartition $\{A, B\}$ having $d$ vertices in Part $A$ and $d^d$ vertices in Part $B$. Let $K_{d, d^d}^{1/2}$ denote the graph obtained by subdividing  every edge of $K_{d, d^d}$ exactly once. Then, (i) $ch^{*}_{ON}(K_{d, d^d}^{1/2}) \leq 2$, (ii) $\chi_{ON}(K_{d, d^d}^{1/2}) \leq 3$, and (iii) $ch_{ON}(K_{d, d^d}^{1/2}) \geq d+1$.
\end{example}

\begin{proof}
Let $A = \{a_1, \ldots , a_d\}$ and $B = \{b_1, \ldots , b_{d^d}\}$. We shall use $v_{i,j}$ to denote the vertex obtained by subdividing the edge $a_ib_j$.
\\ 
(i).  Let $\mathcal{L} = \{L_v~:~v \in V(K_{d, d^d}^{1/2})\}$ be a 2-assignment for $K_{d, d^d}^{1/2}$. For $ 1 \leq i \leq d$, each vertex $v_{i,i}$ is arbitrarily assigned a color from its list. For $d+1 \leq j \leq d^d$, each vertex $v_{d, j}$ is assigned a color from its list that is not equal to the color assigned to $v_{d,d}$. Finally, every vertex in Part $A$ is assigned an arbitrary color from its list. It is easy to see 
that, the above coloring is a valid $\mathcal{L}$-CFON$^*$-coloring of $K_{d, d^d}^{1/2}$. Thus, $ch^*_{ON}(K_{d, d^d}^{1/2}) \leq 2$.
\\ 
(ii). Each vertex in Part $A$ is given Color 1, and each vertex in Part $B$ is given Color 2. For $ 1 \leq i, j \leq d$, each vertex $v_{i,i}$ is assigned Color 1 and each $v_{i,j}$, $i \neq j$, is assigned Color 2. For $d+1 \leq j \leq d^d$, each vertex $v_{d, j}$ is assigned Color 3. Finally, for $1 \leq i \leq d-1$ and $d+1 \leq j \leq d^d$, assign Color 2 to $v_{i,j}$. It is left to the reader to verify that this is indeed a valid CFON coloring. 
\\ 
(iii). Now consider a list CFON coloring of $K_{d, 
d^d}^{1/2}$ that colors every vertex. Such a coloring cannot have a vertex $a_i \in A$ and $b_j \in B$ receive the same color as that leaves the vertex $v_{i,j}$ seeing no unique color in its open neighborhood. Thus, 
$ch_{ON}(K_{d, d^d}^{1/2}) \geq ch(K_{d, d^d})$. It is known that $ch(K_{d, d^d}) \geq d+1$ (see \cite{gravier1996hajos}).  
\end{proof}

\begin{remark}
Let $n$ denote the number of vertices of $K_{d, d^d}^{1/2}$. Then, $n= d+d^d+d^{d+1}$. From Example \ref{ex:partial_vs_full}, $ch_{ON}(K_{d, d^d}^{1/2}) \geq d+1 = \Omega(\frac{\ln n}{\ln\ln n})$ and $ch^*_{ON}(K_{d, d^d}^{1/2}) \leq 2$. Thus, $ch_{ON}(K_{d, d^d}^{1/2})  = \Omega(ch^*_{ON}(K_{d, d^d}^{1/2}) + \frac{\ln n}{\ln\ln n})$. We prove in Theorem \ref{thm:partial_vs_full} that for any hypergraph $\mathcal{H}$ on $n$ vertices, $ch_{CF}(\mathcal{H}) =  O(ch_{CF}^*(\mathcal{H}) + \ln n)$.    
\end{remark}

The following proposition gives graphs whose CFON$^*$ chromatic number and CFON$^*$ choice number are far apart.

\begin{example}
\label{ex:partial_list_vs_partial}    
Let $d>2$. Let $P_{d, d^d}^{1/2}$ denote the graph obtained by adding exactly one pendant vertex to every vertex in $K_{d, d^d}^{1/2}$ whose degree is greater than $2$. Then,  (i) $\chi_{ON}^*(P_{d, d^d}^{1/2}) \leq 3$, and (ii) $ch^*_{ON}(P_{d, d^d}^{1/2}) \geq d+1 = \Omega(\frac{\ln n}{\ln \ln n})$, where $n$ denotes the number of vertices in $P_{d, d^d}^{1/2}$. 
\end{example}

\begin{proof}
We define $A,B, a_i, b_j,$ and $v_{i,j}$ as in Example \ref{ex:partial_vs_full}. \\
(i) Each vertex in Part $A$ is assigned Color 1, each vertex in Part $B$ is assigned Color 2, and finally every pendant vertex is assigned Color 3.\\
(ii) For the pendant vertices to see a unique color in their open neighborhood, it is essential that every vertex in $A$ and $B$ are colored. The arguments used to prove (iii) in Example \ref{ex:partial_vs_full} follow.  
\end{proof}

The following proposition connects the CFON and  CFCN choice numbers of a graph.

\begin{proposition}
\label{l-cn and l-on relation}
For any graph $G$ on $n$ vertices,\\
(i) $ch_{CN}(G) \le 2 ch_{ON}(G) \cdot \ln n + 1$, \\
(ii) $ch_{CN}^*(G) \le 2 ch_{ON}^*(G) \cdot \ln n_1 + 1$, where $n_1$ is the minimum number of colored vertices over all possible CFCN$^*$ colorings of $G$ that use only $\chi_{CN}^*(G)$ colors.
\end{proposition}

\begin{proof}
We shall prove only (i) as the proof of (ii) is similar. 
We have, $ch_{CN}(G) \le \chi_{CN}(G) \cdot \ln n + 1 \le 2 \chi_{ON}(G) \cdot \ln n + 1 \le 2 ch_{ON}(G) \cdot \ln n + 1$, where the first and last inequalities follow from Proposition \ref{ON and l-ON relation} and the second inequality follows from Proposition \ref{prop:CFCNON}.
\end{proof} 

A path on two vertices is a graph $G$ with $2 = ch_{CN}(G) > ch_{ON}(G) = 1$. We do not know of any graph $G$ for which $ch_{CN}(G)/ch_{ON}(G)$ is arbitrarily large. Thus, it is possible that the bound given in Proposition \ref{l-cn and l-on relation} is far from tight.

In the table below, we present a collection of known observations, along with some new ones, in connection with various conflict-free coloring and conflict-free choosability parameters. We also provide tight examples that illustrate the relationships between these parameters. Here, $\mathcal{H}$ denotes a hypergraph on $n$ vertices, and $G$ denotes a graph on $n$ vertices.

\begin{table}[H]
\centering
\begin{tabular}{|p{7cm}|p{7cm}|}
\hline 
\textbf{Relation} & \textbf{Asymptotic Tightness} \\
\hline \hline
 $ \chi_{\mathrm{CF}}(\mathcal{H}) \le ch_\mathrm{CF}(\mathcal{H}) \le   \chi_\mathrm{CF}(\mathcal{H}) \cdot \ln n + 1$ [Known] & Up to multiplicative factor of $\ln \ln n$ [Expl. ~\ref{ex:partial_vs_full}]     
\\ 
\hline 
 $ \chi_\mathrm{CF}^*(\mathcal{H}) \le ch_\mathrm{CF}^*(\mathcal{H}) \le   \chi_\mathrm{CF}^*(\mathcal{H}) \cdot \ln n + 1$ [Known]  & Up to multiplicative factor of $\ln \ln n$ [Expl. ~\ref{ex:partial_list_vs_partial}]       
\\ 
\hline \hline 

\textcolor{gray}{$\chi_\mathrm{CN}(G) \leq 2\chi_\mathrm{ON}(G)$ [Known]} & \textcolor{gray}{Tight  [Known]}       
\\ 
\hline
\textcolor{gray}{$ch_\mathrm{CN}(G) \le 2 ch_\mathrm{ON}(G) \cdot \ln n + 1$ [Prop. ~\ref{l-cn and l-on relation}]}  & \textcolor{gray}{Not known}    
\\ 
\hline \hline
 $ \chi_\mathrm{CF}(\mathcal{H}) \leq \chi^*_\mathrm{CF}(\mathcal{H}) + 1$ [Known] &  Tight  [Known] 
\\ 
\hline
 $ch_\mathrm{CF}(\mathcal{H}) =  O(ch_\mathrm{CF}^*(\mathcal{H}) + \ln n)$ [Thm. ~\ref{thm:partial_vs_full}]   & Up to additive factor of ($\ln n - \frac{\ln n}{\ln \ln n})$  [Expl. ~\ref{ex:partial_vs_full}]        
\\ 
\hline
\end{tabular}
\begin{center}
\caption{A summary of the quick observations}
\end{center}
\label{table_quick_obvs}
\end{table}

\subsection{Preliminaries}
\label{subsec_prelim}
We study simple, finite, and undirected graphs throughout this paper. When discussing open neighborhood coloring, we assume the graph under consideration has no isolated vertices. 
Let \emph{palette} of $\mathcal{L}$, denoted $\mathcal{P}_{\mathcal{L}}$, be defined as $\mathcal{P}_{\mathcal{L}} := \bigcup_{v \in V(G)}L_v$.
Given an $S \subseteq V(G)$, we shall use (i) $G[S]$ to denote the subgraph of $G$ induced by the vertices in $S$, and (ii) $G-S$ to denote the subgraph induced by the vertices in $V(G)\setminus S$. 
The \emph{degree} of an element $v \in V$ in a given hypergraph $\mathcal{H} = (V, \mathcal{E})$, denoted by $d_{\mathcal{H}}(v)$, is the number of hyperedges that $v$ is present in. The \emph{maximum degree} of $\mathcal{H}$ is defined as $\max \{d_\mathcal{H}(v) :v \in V\}$. 
A \emph{planar graph } is a graph that can be drawn on a plane in such a way that no two distinct edges intersect, except at a shared vertex.

\subsection{Our results}
\label{subsec_our_methods}
In this paper, we prove upper bounds for CF$^*$ choice numbers of general hypergraphs and CFON$^*$/CFCN$^*$ choice numbers of general graphs. One can use Theorem \ref{thm:partial_vs_full} to extend these results to CF choice numbers and CFON/CFCN choice numbers. 

Our first result comes from observing that the proof of the general upper bound for the conflict-free chromatic number of a hypergraph due to  Pach and Tardos \cite{pach2009conflict} can be easily extended to partial list CF coloring. We thus show in Section \ref{sec_upper bounds} that for a hypergraph $\mathcal{H},~ ch^*_{CF}(\mathcal{H}) = (t\Gamma^{1/t}\ln \Gamma)$, where every hyperedge in $\mathcal{H}$ is of size at least ${2t - 1}$ and every hyperedge overlaps with at most $\Gamma$ other hyperedges. Let $G$ be a graph with maximum degree $\Delta$ and minimum degree $\Omega(\ln \Delta)$. Applying the above result on the open/closed neighborhood hypergraphs of $G$ with $\Gamma = \Delta^2$ and $t = \frac{\ln\Delta + 1}{2}$, {we show in Section \ref{sec_upper bounds} that (i) ${ch^*_{ON}(G) = O(\ln^2 \Delta)}$, and (ii)  ${ch^*_{CN}(G) = O(\ln^2 \Delta)}$}.

For any graph $G$ with maximum degree $\Delta$, Pach and Tardos \cite{pach2009conflict} in 2009 showed that for any constant $\epsilon > 0$, $\chi_{CN}^*(G) = O(\ln^{2 + \epsilon }\Delta)$.
Later, Glebov, Szabó, and Tardos \cite{glebov2014conflict} in 2014 showed the existence of a graph $G$ with $\chi_{CN}^*(G) = \Omega(\ln^2\Delta)$. Bhyravarapu, Kalyanasundaram, and Mathew \cite{bhyravarapu2021short} in 2021 bridged the gap between the upper and the lower bound. They showed that $\chi_{CN}^*(G) = O(\ln^2 \Delta)$.
We generalize the upper bound of \cite{bhyravarapu2021short} to list CFCN${^*}$ coloring by proving ${ch^*_{CN}(G) = O(\ln^2 \Delta)}$ in Section \ref{sec_upper bounds}. The proof of $\chi_{CN}^*(G) = O(\ln^2 \Delta)$ given in \cite{bhyravarapu2021short} uses the idea of maximal distance-$3$ sets and cannot be adapted to prove the list version. Our proof crucially uses the extension (mentioned in the paragraph above) of the theorem due to Pach and Tardos.

We study the computational complexity of the CFON/CFON$^*$/CFCN choosability problems in Section \ref{Hardness Results}. Let $k$ be a constant integer. We show that the \textsc{$k$-CFON-choosability  problem}, the \textsc{$k$-CFON$^*$-choosability  problem}, and the \textsc{$g^G_{\{2,k\}}$-CFCN-choosability  problem} (defined in Section \ref{subsec:hardness_cfcn}) are $\Pi_2^{P}$-complete on (i) bipartite graphs, for any $k \geq 3$, (ii) planar triangle-free graphs, when $k=3$, and (iii) planar graphs, when $k=4$.

\section{Auxiliary results}
\label{sec:auxiliary_lemmas}

In this section, we state some known results (and a new result, Theorem \ref{thm:partial_vs_full}) that we use later. 

\begin{theorem} \cite{cheilaris2011potential}
\label{max degree + 1}
For any hypergraph $\mathcal{H}$ with maximum degree  $\Delta$, $ch_{CF}(\mathcal{H}) \leq \Delta + 1$. 
\end{theorem}

Below we state the Local Lemma due to Erd\H{o}s and Lov{\'a}sz \cite{lovaszlocallemma}, and a variant of Chernoff Bound that is used in the proofs of Theorem \ref{thm:improve_Pach} and Lemma \ref{partition}, respectively.

\begin{lemma}[\emph{The Local Lemma}, \cite{lovaszlocallemma}] 
\label{lem:local} 
Let $A_1, \ldots, A_n$ be events in an arbitrary probability space. Suppose that each event $A_i$ is mutually independent of a set of all the other events $A_j$ but at most $d$, and that $Pr[A_i] \leq p$ for all $i \in [n]$. If 
$$ep(d+1) \leq 1\;\;\; or \;\;\; 4pd \leq 1\;,$$ 
then $Pr[\land _{i=1}^n \overline{A_i}] > 0$.  
\end{lemma}

\begin{theorem}
[Chernoff Bound, Corollary 4.6 in \cite{mitzenmacher}]
\label{thm_Chernoff}
Let $X_1, \ldots , X_n$ be independent Poisson trials such that $Pr[X_i =1] = p_i$. Let $X = \sum_{i=1}^n X_i$ and $\mu = E[X]$. For $0 < \delta < 1$, 
$Pr[|X-\mu| \geq \delta \mu] \leq 2e^{-\mu \delta^2/3}$. 
\end{theorem}

We prove the following lemma which states that for a given hypergraph $\mathcal{H}$ and a list assignment for $\mathcal{H}$, partitions the palette into $t$ parts such that for each vertex $v$, the intersection of its list $L_v$ with each part of the partition is reasonably large. This lemma is used in the proof of Theorem \ref{thm:partial_vs_full}.

\begin{lemma}
\label{partition}
Let $\mathcal{H}$ be a hypergraph on $n$ vertices. Let $z, t$ be two positive integers, where $t \ge 2$. Let $\mathcal{L} = \{L_v~:~ v \in V(\mathcal{H})\}$ be an $r$-assignment for $\mathcal{H}$, where $r = 5t (z + \lceil \ln n \rceil)$. For every $v \in V(\mathcal{H})$, let $i_v$ be an integer that can take any value from $1$ to $t$. Then, there exist a partitioning of $\mathcal{P}_{\mathcal{L}}$ into $t$ parts, namely $\mathcal{P}_{\mathcal{L}}^1, \mathcal{P}_{\mathcal{L}}^2, \ldots, \mathcal{P}_{\mathcal{L}}^t$, that satisfy the following property. Let $L_v^j := L_v \cap \mathcal{P}_{\mathcal{L}}^j$.  Then, $\forall v \in V(\mathcal{H})$, $9 (z + \lceil \ln n \rceil) \ge |L_v^{i_v}| \ge z + \lceil \ln n \rceil $.
Thus, $\mathcal{L}' := \{L_v^{i_v} ~:~ v \in V(G) \}$ is a $(z + \lceil \ln n \rceil)$-assignment for $\mathcal{H}$. 
\end{lemma}

\begin{proof}
We first partition the palette of colors  $\mathcal{P}_{\mathcal{L}}$ into $t$ parts, say $\mathcal{P}_{\mathcal{L}}^1, \mathcal{P}_{\mathcal{L}}^2 \dots, \mathcal{P}_{\mathcal{L}}^t$. 
For each color $c \in  \mathcal{P}_{\mathcal{L}}$,  independently and uniformly at random, choose an integer $i$ from $\{1, \ldots , t\}$ and add $c$ to $\mathcal{P}_{\mathcal{L}}^i$. 
Fix a vertex $v$. Let $c \in L_v$.  Let $X_v^c$ be a $0$-$1$ indicator random variable which takes the value $1$ if and only if color $c$ is present in $L_v^{i_v}$. Let $X_v := \sum_{c \in L_v} X_v^c$. Then, $\mu_v : = \mathbb{E}[X_v] = \frac{1}{t} \cdot 5t (z + \lceil \ln n \rceil ) = 5 (z + \lceil \ln n \rceil )$. 
Applying the Chernoff bound given in Theorem \ref{thm_Chernoff} with $\delta = \sqrt{3/5}$, $Pr[|X_v -\mu_v| \geq \delta \mu_v] \leq 2e^{-\mu_v \delta^2/3} = 2e^{-(z + \lceil \ln n \rceil )} \leq \frac{2}{ne^z}$. Applying union bound over all vertices, the probability that there exists a vertex $v$ with $|X_v -\mu_v| \geq \frac{\sqrt{3} \mu_v}{\sqrt{5}}$ is strictly less than $1$. Thus, with a positive probability, $\forall v \in V(\mathcal{H}), 9 (z + \lceil \ln n \rceil) \ge |L_v^{i_v}| \ge z + \lceil \ln n \rceil$. 
\end{proof}

For any $u,v \in V(\mathcal{H})$, since $i_u \neq i_v$ implies $L_u^{i_u} \cap L_v^{i_v} = \emptyset$, we have the following observations  regarding Lemma \ref{partition}.

\begin{observation}
\label{partition obs}
Let $u,v \in V(\mathcal{H})$ such that $i_u \neq i_v$. Then in every $\mathcal{L}'$-coloring of $\mathcal{H}$, $u$ and $v$ receive distinct colors.  
\end{observation}

\begin{observation}
\label{partition obs 2}
Let $u,v \in V(\mathcal{H})$ and $i, j \in \{1, \ldots t\}$ such that $i \neq j$. Then, $L_u^i \cap L_v^j = \emptyset$.    
\end{observation}

\noindent Below, we state the theorem that gives an upper bound on the CF choice number of a hypergraph in terms of its CF$^*$ choice number. Thus, all the bounds that we prove in this paper for CF$^*$/CFON$^*$/CFON$^*$ choice numbers can be extended to CF/CFON/CFON choice numbers.

\begin{theorem}
\label{thm:partial_vs_full}
    For any hypergraph $\mathcal{H} = (V , \mathcal{E})$ with $|V| = n$, $ch_{CF}(\mathcal{H}) =  O(ch_{CF}^*(\mathcal{H}) + \ln n)$.
\end{theorem}
\begin{proof}
Let $z := ch_{CF}^*(\mathcal{H})$ and let $\mathcal{L} = \{L_v ~:~v \in V\}$ be an $r$-assignment for $\mathcal{H}$, where $r = 10(z + \lceil \ln n\rceil )$. Applying Lemma \ref{partition} with $t= 2$ and $i_v = 1$, $~\forall v \in V$, we get a $(z + \lceil \ln n \rceil)$-assignment $\mathcal{L}'$ for $\mathcal{H}$. Consider an $\mathcal{L}'$-CF$^*$-coloring of $\mathcal{H}$. Let $U$ denote the set of uncolored vertices in this coloring. For each $u \in U$, assign any color from the list $L_u^2$ (this list is non-empty as $L_u = L_u^1 \uplus L_u^2$, where $|L_u| = 10(z + \lceil \ln n\rceil )$ and $|L_u^1| = |L_u^{i_u}| \leq 9(z + \lceil \ln n\rceil)$ (by Lemma \ref{partition}) ). By Observation \ref{partition obs 2}, the coloring thus obtained will continue to maintain its conflict-freeness property. 
\end{proof}

The bound in Theorem \ref{thm:partial_vs_full} is asymptotically tight due to the Example \ref{ex:partial_vs_full}.

\section{Upper bounds}
\label{sec_upper bounds}
The following theorem is an extension of Theorem 1.2 in \cite{pach2009conflict} from CF coloring to list CF coloring of hypergraphs. The proof below is also a straightforward extension of the proof in \cite{pach2009conflict}.  

\begin{theorem}
\label{thm:improve_Pach}
For any positive integers $t$ and $\Gamma$, the partial list conflict-free chromatic number of any hypergraph $\mathcal{H}=(V,\mathcal{E})$ in which each edge is of size at least $2t - 1$ and intersects at most $\Gamma$ other hyperedges is $O(t \Gamma^{1/t} \ln \Gamma)$.
\end{theorem}

\begin{proof} 
Suppose $t \ge \Gamma$. Then, the bound we get is at best $O(\Gamma^{1+(1/t)} \ln \Gamma)$. Let $\Delta$ denote the maximum degree of $\mathcal{H}$. It is easy to see that $\Gamma\ge \Delta$. So the bound we get is weaker and therefore follows from the upper bound of $\Delta+1$ given by Theorem \ref{max degree + 1}. For the rest of the proof, we assume that $t < \Gamma$.
Let $\mathcal{L} = \{L_v~:~v \in V\}$, where each list $L_v$ ($\subseteq \mathbb{N}$) is sorted in increasing order. Consider a hyperedge $f$ in $\mathcal{H}$. We know $|f| \geq  2t-1$. We color the points in $f$ using the following iterative process until all the points in $f$ are colored. In Round $i$, each uncolored point then is independently assigned the $i$-th color in its list with a probability $q$ (to be determined later). 

\begin{lemma} 
\label{lem:prob_estimation_of_bad_event} For any hyperedge $f$, $Pr[f \mbox{ is not CF colored from }\mathcal{L}] \leq 2(eqt)^t.$  
\end{lemma}

\begin{proof}
We split the proof into two cases.
\\
\textbf{Case 1: $|f| = 2t-1$.} 
\\ 
We claim that, given a partition of the points in $f$ into $k$ parts, the probability that the coloring of the points of $V$ induces the same partition of $f$ is at most $q^{2t-1-k}$. Treating the first point that gets colored in each part as a \emph{leader}, it is about the probability that the leaders get distinct colors and the rest of the points in each part get the color of its leader. The probability that the rest of the points in each part get the color of its leader given every leader gets a distinct color is at most $q^{2t-1-k}$. We know that the number of ways to partition a $(2t-1)$-sized set into $k$ parts is at most $\frac{k^{2t-1}}{k!} \leq e^k k^{2t-1-k}$. The probability that $f$ is colored with exactly $k$ colors is at most $A_k := e^k (qk)^{2t-1-k}$. We thus have, 
\\
$Pr[f \mbox{ is not $\mathcal{L}$-CF-colored}]  \leq Pr[f \mbox{ is colored with at most $t - 1$ colors}] \leq \sum\limits_{k=1}^{t-1} A_k \leq  2 A_{t-1} \\ \leq (eqt)^t$.
\\
\textbf{Case 2: $|f| > 2t-1$.}
\\ 
Let $f'$ denote a set of $2t$ points in $f$ such that every point in $f'$ got a color that is greater than or equal to the color of every point in $f\setminus f'$. If $f$ is not $\mathcal{L}$-CF-colored, then there is at most one color that appears exactly once in $f'$. Given a partition of the points in $f'$ into $k$ parts, the probability that the coloring of the points of $V$ induces the same partition of $f'$ is at most $q^{2t-k}$ (the probability calculation is same as the one in Case 1). The number of ways to partition a $(2t)$-sized set into $k$ parts is at most $e^k k^{2t-k}$. The probability that $f'$ is colored with exactly $k$ colors is at most $B_k := e^k (qk)^{2t-k}$. Thus, 
\\
$Pr[f \mbox{ is not $\mathcal{L}$-CF-colored}] \leq Pr[f' \mbox{ is colored with at most t colors}] \leq \sum\limits_{k=1}^{t} B_k \leq  2 B_{t} \leq  2 (eqt)^t$.
 \end{proof}

We first trim the hypergraph $\mathcal{H}=(V,\mathcal{E})$ as described here. We designate some $2t-1$ points for each hyperedge in $\mathcal{H}$. Let $V'$ be obtained from $V$ by removing all those points in $V$ that are not designated to any hyperedge. Let $\mathcal{H'} = (V',\mathcal{E'})$ be the hypergraph thus obtained, where $\mathcal{E'}= \{f\cap V'~:~f \in \mathcal{E}\}$. We observe that for each $f \in \mathcal{E'}$, $2t-1 \leq |f| \leq (2t-1)(\Gamma + 1) < 2 \Gamma^2$ (since we assume $t < \Gamma$). 

We now color the points in $V'$ using the iterative process described above.  In Round $i$, each uncolored point then in $V'$ is independently assigned the $i$-th color in its list with probability $q$. Let $q=\frac{1}{30t\Gamma^{1/t}}$. For a hyperedge $f \in \mathcal{E'}$, let $A_f$ denote the event that $f$ is not CF-colored from $\mathcal{L}$ after the iterative process is over. From Lemma \ref{lem:prob_estimation_of_bad_event}, we know that  $Pr[A_f] \leq \frac{1}{5\Gamma}$. Let $T= O(t \Gamma^{1/t}\ln \Gamma )$. Let $B_f$ denote the event that some point in $f$ is not colored after $T$ rounds of the iterative coloring. We have $Pr[B_f] \leq |f| \cdot (1-q)^T < 2 \Gamma^2(1-q)^T \leq \frac{1}{20\Gamma^3}$. Let $C_f$ be the bad event that  $A_f$  or $B_f$ occurs. Then, $Pr[C_f] \leq \frac{1}{e(\Gamma + 1)}$. Note that $C_f$ is independent from all events $C_g$ for edges $g$ that are disjoint from $f$. Applying Lemma \ref{lem:local} (the Local Lemma) with $p = \frac{1}{e(\Gamma + 1)}$ and $d = \Gamma$, we conclude that there is a non-zero probability that $\mathcal{H}'$ is $\mathcal{L}$-CF-colored, and thereby $\mathcal{H}$ is $\mathcal{L}$-CF$^*$-colored.  
\end{proof}

\begin{corollary}
\label{min degree log delta}
     Let $G$ be a graph with maximum degree $\Delta$ and minimum degree $\Omega(\ln \Delta)$. We have,
   (i) $ch^*_\mathrm{ON}(G) = O(\ln^2 \Delta)$. 
   (ii)  $ch^*_\mathrm{CN}(G) = O(\ln^2 \Delta)$.
   
\end{corollary}

\begin{proof}
Consider the open/closed neighborhood hypergraph of $G$. Every hyperedge in this hypergraph is of size at least $2t -1$ and overlaps with $\Gamma$ other hyperedges, where $t = \Omega(\ln \Delta)$ and $\Gamma = \Delta^2$. Apply Theorem \ref{thm:improve_Pach} to obtain the desired result. 
\end{proof}

Below we prove the main result of this section.

\begin{theorem}
\label{thm: general cfcn upper bound}
For any graph $G$ with maximum degree $\Delta$, $ch^*_{CN}(G) = O(\ln^2 \Delta)$.
\end{theorem}

\begin{proof}
We first partition $V(G)$ into two parts. Let $A : = \{v \in V(G) ~:~ d_G(v) \ge \ln \Delta \}$ and $B : = V(G) \setminus A$. We further partition $A$ into two components, namely $A_1$ and $A_2$. Let $A_1: = \{v \in A~:~ d_A(v) \ge \ln \Delta \}$ and $A_2: = A \setminus A_1$, where $d_A(v)$ denotes the number of neighbors of $v$ in $A$. Thus every vertex in $A_2$ has some neighbor in $B$, whereas a vertex in $A_1$ may or may not have a neighbor in $B$.

Let $\mathcal{L} = \{L_v~:~v \in V(G)\}$ be a $r$-assignment for $G$ given by the adversary, where $r = K\cdot \lceil \ln^2 \Delta \rceil$ and $K$ is a sufficiently large integer constant.  We obtain the desired $\mathcal{L}$-CFCN$^*$-coloring of $G$ by list CF$^*$-coloring  hypergraphs $\mathcal{H}_1, \mathcal{H}_2$ that are defined below. A list CF$^*$-coloring of the hypergraph $\mathcal{H}_1$ will ensure that every vertex in $A_1$ sees a unique color in its closed neighborhood. In a similar way, a list CF$^*$-coloring of the hypergraph $\mathcal{H}_2$ will ensure that every vertex in $A_2 \cup B$ sees a unique color in its closed neighborhood. Before coloring $\mathcal{H}_2$, we update the lists of its vertices to avoid `conflicts'. We explain this in detail below.   
\begin{itemize}
    \item Let $\mathcal{H}_1 = (V_1, \mathcal{E}_1)$ be a hypergraph, where $V_1 = A$ and $\mathcal{E}_1 = \{N_G[v] \cap A~ : ~v \in A_1 \}$. From the definition of $A_1$, the minimum size of a hyperedge in $\mathcal{H}_1$ is at least $ \ln \Delta$. Each hyperedge in $\mathcal{H}_1$ overlaps with at most $\Delta^2$ other hyperedges. Hence, by Theorem \ref{thm:improve_Pach}, we have  $ch^*_{CF}(\mathcal{H}_1) = O(\ln^2 \Delta)$. Let $f_1$ denote this coloring. Under $f_1$, every vertex in $A_1$ sees a unique color in its closed neighborhood. 
\end{itemize}
Next, by list CF$^*$ coloring a hypergraph $\mathcal{H}_2$ (defined below) having vertex set $B$, we intend to take care of every vertex in $A_2 \cup B$. This however can lead to problems of two types which are described below. 
\\ (i) Let $c_v$ be the unique color seen by a vertex $v \in A_1$ under the coloring $f_1$. We should ensure that none of the neighbors of $v$ in $B$ receive the color $c_v$. For this, we update the lists of vertices $b \in B$ as $L_b^1 = L_b \setminus S_b$, where $S_b = \{c_v~:~v \in N_G(b) \cap A_1\}$. Note that $|S_b| < \ln \Delta$. 
\\ (ii) For every $w \in A_2 \cup B$, let $c_w$ be the unique color to be seen by $w$ (at some vertex $b \in B$) under the list CF$^*$ coloring of $\mathcal{H}_2$ (to be defined) having vertex set $B$. We need to ensure that such a vertex $w$ does not have a neighbor in $A$ which received the color $c_w$ under the coloring $f_1$. For this purpose, for each vertex $b \in B$, we define $T_b = \{f_1(a)~:~a \in N_G[w] \cap A, w \in N_G[b] \cap (A_2 \cup B)\} $. Since $|N_G[b]| \leq \ln \Delta$ and $|N_G[w] \cap A| \leq \ln \Delta$, for all $w \in (A_2 \cup B)$, we have $|T_b| \leq \ln^2\Delta$. Let $L_b^2 = L_b^1 \setminus T_b$, for every $b \in B$. By taking a sufficiently large integer constant $K$ (recall that,  every list $L_v$ is of size $K\cdot \lceil 
\ln^2 \Delta \rceil$), we can ensure that $|L_b^2| \geq \ln \Delta + 1$.  
\begin{itemize}
    \item Let $\mathcal{L}^2 = \{L_b^2~:~b \in B\}$. Let $\mathcal{H}_2 = (V_2, \mathcal{E}_2)$, where $V_2 = B$ and $\mathcal{E}_2 = \{N_G[v] \cap B ~: ~v \in A_2 \cup B\}$.  From the definition of $B$, the maximum degree of $\mathcal{H}_2$ is at most $ \ln \Delta$. Hence, by Theorem \ref{max degree + 1}, $\mathcal{H}_2$ admits an $\mathcal{L}^2$-CF$^*$-coloring. 
\end{itemize}    
This completes the proof. 
\end{proof}

\section{Hardness Results}
\label{Hardness Results}

\subsection{Conflict-free open neighborhood choosability}
The only connected  graph with a CFON-choice-number equal to 1 is a path on 2 vertices. Therefore, deciding whether a
graph is 1-CFON choosable can be done in polynomial time.

Let $k \geq 3$ be an integer. We now show the $\Pi_2^{P}$-hardness of the \textsc{$k$-CFON-choosability problem} and the \textsc{$k$-CFON$^*$-choosability problem} by reducing from the \\ \textsc{$k$-choosability problem} (which is defined below). See section \ref{sec list Cf coloring} for the definition of $ k$-choosability.

\defproblem{\textsc{$k$-choosability problem ($k$-CH problem)}}{A graph $G$ and an integer $k \geq 3$.}{Is $G$ $k$-choosable?}

\defproblem{\textsc{$k$-CFON$^*$-choosability problem ($k$-CFON$^*$-CH problem)}}{A graph $G$ and a positive integer $k$.}{Is $G$ $k$-CFON$^*$-choosable?}

\defproblem{\textsc{$k$-CFON-choosability problem ($k$-CFON-CH problem)}}{A graph $G$ and a positive integer $k$.}{Is $G$ $k$-CFON-choosable?}

\begin{theorem}
\label{thm:choosability_known_results}
The \textsc{$k$-choosability  problem} is $\Pi_2^{P}$-complete on:
\begin{enumerate}
\item bipartite graphs, for any constant $k \geq 3$ (see \cite{gutner2009some}).
\item planar triangle-free graphs, when $k=3$ (see \cite{gutner1996complexity}).
\item planar graphs, when $k=4$ (see \cite{gutner1996complexity}).
\end{enumerate}
\end{theorem}


It is easy to see that the \textsc{$k$-CFON-CH problem} is in $\Pi_2^{P}$. Given a graph $G$, a $k$-assignment $\mathcal{L}= \{L_v ~:~ v \in V(G)\}$, and a coloring function $f: V(G) \longrightarrow \bigcup_{v \in V(G)} L_v$, a polynomial time verifier verifies the following (i) every vertex $v \in V(G)$ is assigned a color from its list, i.e., $f(v) \in L(v)$, (ii) every vertex has some unique color in its open neighborhood. In a similar way, we can show that \textsc{$k$-CFON$^*$-CH problem} is in $\Pi_2^{P}$.

Below, we present two constructions, one for a graph $H_G$ and another for a graph $H'_G$, both derived from a given graph $G$. 

\paragraph{Construction of $H_G$:} We construct $H_G$ by subdividing each edge of $G$ exactly once. In other words, for every edge $ uv \in E(G)$ with the endpoints $u$ and $v$, we introduce a new vertex named $x_{uv}$, such that $N_{H_G}(x_{uv}) = \{u, v\}$. 
Note that the vertex $x_{uv}$ is also known as $x_{vu}$.
The vertices of $G$ are called \emph{original vertices} in $H_G$, and the remaining vertices in $H_G$ are called \emph{subdivided vertices}.
It is easy to see that $H_G$ can be constructed from $G$ in $O(|V(G)| + |E(G)|)$ time.

\paragraph{Construction of $H'_G$:} We first construct the graph $H_G$ as described above. Then, we construct the graph $H'_G$ from the graph $H_G$ by attaching a pendant vertex to each original vertex in $H_G$.  We leave the subdivided vertices untouched. Thus, $H'_G$ has three types of vertices: original vertices (which are vertices of $G$), subdivided vertices (which are the newly introduced vertices in $H_G$), and pendant vertices (which are the newly introduced vertices in $H'_G$). It is easy to see that $H'_G$ can be constructed from $H_G$ in $O(|V(H_G)|)$ time.

We make the following observations concerning graphs $H_G$ and $H'_G$.

\begin{observation}
\label{obs subdivided}
For each subdivided vertex $x_{uv}$ to see a unique color in its open neighborhood, the vertices $u$ and $v$ should receive distinct colors in $H_G$.  
\end{observation} 

\begin{observation}
\label{obs subdivided pen}
Let $p_i$ be the pendant vertex attached to $v_i$ in $H'_G$. For each pendant vertex $p_i$ to see a unique color in its open neighborhood, the vertex $v_i$ must be colored in $H'_G$. 
\end{observation} 

\begin{observation}
\label{obs:H_G-property_preserve}
Both $H_G$ and $H'_G$ preserve graph properties such as bipartiteness, triangle freeness, and planarity.
\end{observation} 

We now prove the following lemmas.

\begin{lemma}
\label{lem: subdivision}
A graph $G$ is $k$-choosable if and only if the graph $H_G$ is $k$-CFON-choosable, where $k \ge 3$ is an integer.
\end{lemma}

\begin{proof}
Let $n$ denote $|V(G)|$. Suppose $G$ is $k$-choosable. We will show that $H_G$ is $k$-CFON-choosable. Let $\mathcal{L} = \{L_v~:~v \in V(H_G)\}$ be a $k$-assignment for $H_G$. Let $\mathcal{L}_G$ denote $\mathcal{L}$ restricted to the original vertices in $H_G$. Let $f$ be a proper $\mathcal{L}_G$-coloring of $G$. We now obtain an  $\mathcal{L}$-CFON coloring $f'$ for $H_G$ as follows. For each original vertex $v \in V(H_G)$, let $f'(v) = f(v)$.
 By Observation \ref{obs subdivided}, each subdivided vertex sees a unique color in its open neighborhood. Now, we have to handle the needs of the original vertices in $H_G$. With each original vertex $v \in V(H_G)$, we associate a subdivided vertex $x_{uv}$, where $u \neq v$. We assign some color to $x_{uv}$ from its list. Finally, we color all the so far uncolored subdivided vertices $x_{pq} \in V(H)$ in such a way that $x_{pq}$ does not receive the unique color of the vertex $p$ or the vertex $q$. It is left to the reader to verify that this is indeed a valid CFON coloring of $H_G$. 

Now we prove that if $H_G$ is $k$-CFON-choosable, then $G$ is $k$-choosable. By Observation \ref{obs subdivided}, in any valid CFON coloring of $H_G$, the two neighbors of each subdivided vertex must have different colors. Specifically, every pair of original vertices in $H_G$ that have an edge between them in $G$ must have different colors. This ensures that a CFON coloring induced on original vertices in $H_G$ is a proper coloring of $G$.
\end{proof}

\begin{lemma}
\label{lem: subdivision pen}
A graph $G$ is k-choosable if and only if the graph $H'_G$ is k-CFON$^*$-choosable, where $k \ge 3$ is an integer.
\end{lemma}

\begin{proof}
Since a pendant vertex is attached to every original vertex in $H'_G$,  Observation \ref{obs subdivided pen} forces us to color every original vertex in any valid CFON$^*$ coloring of $H'_G$. Once we keep this observation in mind, the rest of the proof is similar to the proof of Lemma \ref{lem: subdivision} and hence we omit the proof. 
\end{proof}

Combining Theorem \ref{thm:choosability_known_results}, Observation \ref{obs:H_G-property_preserve}, Lemmas \ref{lem: subdivision} and \ref{lem: subdivision pen}, we have the following theorem. 
\begin{theorem}
\label{thm 3 cfon cfon* chhosability}
Both the \textsc{$k$-CFON-choosability  problem} and the \\ \textsc{$k$-CFON$^*$-choosability  problem} are $\Pi_2^{P}$-complete on: 
\begin{enumerate}
\item bipartite graphs, for any constant $k \geq 3$.
\item planar triangle-free graphs, when $k=3$.
\item planar graphs, when $k=4$.
\end{enumerate}
\end{theorem}

\subsection{Conflict-free closed neighborhood choosability}
\label{subsec:hardness_cfcn}
The only graph with a CFCN-choice-number equal to 1 is the edgeless graph, i.e., a graph consisting solely of isolated vertices. Therefore, deciding whether a graph is 1-CFCN-choosable can be done in polynomial time. 

Let $S$ be a finite subset of positive integers. Given a graph $G$, let $g^G_S$ be a function whose domain is $V(G)$ and co-domain is $S$. That is, $g^G_S:V(G)\rightarrow S$. A list assignment $\mathcal{L} = \{ L_v~:~v \in V(G)\}$ for $G$ is a \emph{$g^G_S$-assignment for $G$} if  $ \forall v \in V(G)$, $|L_v| = g^G_S(v)$. For a given function $g^G_S$, we say that $G$ is \emph{$g^G_S$-CFCN-choosable} if for every $g^G_S$-assignment $\mathcal{L}$, $G$ is $\mathcal{L}$-CFCN-colorable. Thus, a $k$-assignment for $G$ is a $g^G_{\{k\}}$-assignment, and the statement `$G$ is $k$-CFCN-choosable' is equivalent to the statement `$G$ is $g^G_{\{k\}}$-CFCN-choosable'.  For a given function  $g^G_{\{2,k\}}$, we define the \textsc{$g^G_{\{2,k\}}$-CFCN-choosability problem} below. 



\defproblem{\textsc{$g^G_{\{2,k\}}$-CFCN-choosability problem ($g^G_{\{2,k\}}$-CFCN-CH problem)}}{A graph $G$, a function $g^G_{\{2,k\}}: V(G)\rightarrow \{2,k\}$.}{Is $G$ $g^G_{\{2,k\}}$-CFCN-choosable?}

\noindent It is easy to see that the \textsc{$g^G_{\{2,k\}}$-CFCN-CH problem} is in $\Pi_2^{P}$. The input is a graph $G$ and a function $g^G_{\{2,k\}}:V(G)\rightarrow \{2,k\}$. Given a $g^G_{\{2,k\}}$-assignment $\mathcal{L}= \{L_v ~:~ v \in V(G)\}$, and a coloring function $f: V(G) \longrightarrow \bigcup_{v \in V(G)}L_v$, a polynomial time verifier verifies the following (i) every vertex $v \in V(G)$ is assigned a color from its list, i.e., $f(v) \in L(v)$, (ii) every vertex has some unique color in its closed neighborhood.

Let $k \geq 3$ be an integer constant. Let $G$ be a graph on $n$ vertices. Below, we present the construction of a graph $H^k_G$ from $G$. 

\paragraph{Construction of $H^k_G$:} Let $N = 2{nk \choose 2}$.  We construct $H^k_G$ by attaching $N$ distinct pendants to each vertex in $G$. Thus, $H^k_G$ has $n+ Nn$ vertices in total. The vertices of $G$ are called \emph{original vertices} in $H^k_G$, and the remaining vertices in $H^k_G$ are called \emph{pendant vertices}.
It is easy to see that $H^k_G$ can be constructed from $G$ in $O(n^3)$  time as $k$ is a constant.

\begin{observation}
\label{obs:H_^k_G-property_preserve}
$H^k_G$ preserves graph properties like bipartiteness, triangle-freeness, and planarity.
\end{observation} 
We now prove the following lemma.

\begin{lemma}
\label{lem:(2,k)-choosability-reduction}
$G$ is $k$-choosable iff $H^k_G$ is $g^G_{\{2,k\}}$-CFCN choosable, where 
$$
g^G_{\{2,k\}}(v) = 
\begin{cases}
2, \mbox{ if $v$ is a pendant vertex in } H^k_G\\
k, \mbox{ otherwise}
\end{cases}
$$
\end{lemma}

\begin{proof}
Suppose $G$ is $k$-choosable. We are given a function $g^G_{\{2,k\}}: V(H^k_G) \rightarrow \{2,k\}$ that satisfies $g^G_{\{2,k\}}(v) = 2$, if $v$ is a pendant vertex in $H^k_G$ and $g^G_{\{2,k\}}(v) = k$, otherwise. Let $\mathcal{L} = \{L_v~:~v \in V(H^k_G)\}$ be a $g^G_{\{2,k\}}$-assignment for $H^k_G$. Let $\mathcal{L}_G$ be $\mathcal{L}$ restricted to the original vertices in $H^k_G$. Let $f$ be a proper $\mathcal{L}_G$-coloring of $G$. We now obtain an $\mathcal{L}$-CFCN coloring $f'$ for $H^k_G$ as follows. For each original vertex $v \in V(H^k_G)$, let $f'(v) = f(v)$. For each pendant vertex $p \in V(H^k_G)$, $f'(p)$ can be any color (from the list $L_p$) other than the color received,  under $f'$, by $p$'s only neighbor. It can be verified that $f'$ is a valid $\mathcal{L}$-CFCN coloring of $H^k_G$.   

Suppose $H^k_G$ is $g^G_{\{2,k\}}$-CFCN choosable, where $g^G_{\{2,k\}}(v) = 2$ if and only if $v$ is a pendant vertex in $H^k_G$. Let $\mathcal{L} = \{L_v~:~v \in V(G)\}$ be a $k$-assignment for $G$. Let palette of $\mathcal{L}$ be defined as $\mathcal{P}_{\mathcal{L}} := \bigcup_{v \in V(G)}L_v$.  Clearly, $P := |\mathcal{P}_{\mathcal{L}}| \leq nk$. Below, we define a $g^G_{\{2,k\}}$-assignment $\mathcal{L}' = \{L'_v~:~v \in V(H^k_G)\}$ for $H^k_G$. For each original vertex $v \in H^k_G$, we have $L'_v = L_v$. Recall that, each original vertex has $N$ (which is at least  $2{P \choose 2}$) pendants attached to it. For the pendant vertices, we assign $2$-sized lists that are subsets of $\mathcal{P}_{\mathcal{L}}$ in such a way that for every 2-sized subset $S$ of $\mathcal{P}_{\mathcal{L}}$ and for every original vertex $v \in V(H^k_G)$, at least two pendant vertices attached to $v$ have $S$ as its list under $\mathcal{L}'$. Thus, $\mathcal{P}_{\mathcal{L}'} = \mathcal{P}_{\mathcal{L}}$. Let $f'$ be an  $\mathcal{L}'$-CFCN coloring of $H^k_G$. Recall that $L'_v = L_v$, for all original vertices $v \in V(H^k_G)$. We claim that $f'$ restricted to the original vertices of $H^k_G$ is an $\mathcal{L}$-coloring of $G$. In other words, under $f'$, every original vertex in $H^k_G$ sees its own color as the unique color in its closed neighborhood. This is because it sees every other color in the palette $\mathcal{P}_{\mathcal{L}}$ at least twice among the pendants attached to it. 
\end{proof}

Combining Theorem \ref{thm:choosability_known_results}, Lemma \ref{lem:(2,k)-choosability-reduction}, and Observation \ref{obs:H_^k_G-property_preserve}, we have the following theorem.  
\begin{theorem}
\label{thm 23 cfcn chhosability}
The \textsc{$g^G_{\{2,k\}}$-CFCN-choosability  problem} is $\Pi_2^{P}$-complete on:
\begin{enumerate}
\item bipartite graphs, for any constant $k \geq 3$.
\item planar triangle-free graphs, when $k=3$.
\item planar graphs, when $k=4$.
\end{enumerate}
\end{theorem}

\section{Concluding remarks}
We have established hardness results for $k$-CFON/CFON$^*$-choosability and for $(2,k)$-CFCN-choosability, for all $k \ge 3$. However, the complexity of the following problems remain open : 
(i) \textsc{$k$-CFON$^*$-choosability} when $k =1, 2$,
(ii) \textsc{$k$-CFON-choosability} when $k =2$,
(iii) \textsc{$k$-CFCN$^*$-choosability} when $k \ge 1$, and
(iv) \textsc{$k$-CFCN-choosability} when $k \ge 2$.

\bibliographystyle{plain}
\bibliography{Thesis_reference}
\end{document}